\documentclass[12pt, reqno]{amsart}
\usepackage{amssymb, amsthm, amsmath, amsfonts}
\usepackage{array, epsfig}
\usepackage{bbm}
\usepackage{hyperref}
\usepackage[numbers,square]{natbib}
\usepackage{color}

  \evensidemargin
\oddsidemargin
\newcommand{\cP}{\mathcal{P}}

\newcommand{\E}{\mathbb E\,}

\newcommand{\R}{\mathbb{R}}

\renewcommand{\P}{\mathbb{P}}

\newcommand{\Span}{\mathop{\mathrm{lin}}\nolimits}
\newcommand{\Int}{\mathop{\mathrm{Int}}\nolimits}

\newcommand{\conv}{\mathop{\mathrm{conv}}\nolimits}
\newcommand{\pos}{\mathop{\mathrm{pos}}\nolimits}

\newcommand{\dd}{{\rm d}}
\newcommand{\eee}{{\rm e}}

\theoremstyle{plain}
\newtheorem{theorem}{Theorem}[section]
\newtheorem{lemma}[theorem]{Lemma}
\newtheorem{corollary}[theorem]{Corollary}
\newtheorem{proposition}[theorem]{Proposition}

\theoremstyle{definition}

\theoremstyle{remark}
\newtheorem{remark}[theorem]{Remark}

\begin{document}

\author{Zakhar Kabluchko}
\address{Zakhar Kabluchko: Institut f\"ur Mathematische Stochastik,
Westf\"alische Wilhelms-Universit\"at M\"unster,
Orl\'eans--Ring 10,
48149 M\"unster, Germany}
\email{zakhar.kabluchko@uni-muenster.de}

\author{Dmitry Zaporozhets}
\address{Dmitry Zaporozhets: St.\ Petersburg Department of Steklov Mathematical Institute,
Fontanka~27,
191011 St.\ Petersburg,
Russia}
\email{zap1979@gmail.com}

\title{Angles of the Gaussian simplex}



\begin{abstract}
Consider a $d$-dimensional simplex whose vertices are random points chosen independently according to the standard Gaussian distribution on $\mathbb R^d$. We prove that the expected angle sum of this random simplex equals the angle sum of the regular simplex of the same dimension $d$.
\end{abstract}

\maketitle

\section{Main result}
The sum of measures of angles in  any triangle in the Euclidean plane is constant. However, a similar statement is not true in higher dimensions. The sum of solid angles of a $d$-dimensional simplex, where $d\geq3$, can take any value between $0$ and $1/2$ of the full solid angle.  This, and more general results, were obtained in the works of H{\"o}hn~\cite{hoehn}, Gaddum~\cite{gaddum1,gaddum2}, Perles and Shephard~\cite[(24) on pp.~208--209]{perles_shephard},  Barnette~\cite{barnette},  Feldman and Klain~\cite{feldman_klain}. We shall give yet another proof  in Proposition~\ref{730}. Knowing the upper and lower bounds, it is natural to ask about the ``average value'' of the sum of solid angles of the simplex.  Of course, the notion of ``average'' depends on the probability measure we put on the set of all simplices. In the present paper, we consider the \emph{Gaussian simplex}, i.e.\ a random simplex in $\R^d$ whose vertices $X_0,\ldots,X_d$ are chosen independently according to the standard Gaussian distribution on $\R^d$. Our main result is the following


\begin{theorem}\label{theo:main}
The expected sum of the solid angles of the Gaussian simplex coincides with the sum of the solid angles of the regular simplex of the same dimension.
\end{theorem}

Let us mention some related results. Feldman and Klain~\cite{feldman_klain} showed that in  every tetrahedron, the sum of solid angles, measured in steradiants and divided by $2\pi$, gives the probability that a random projection of the tetrahedron onto a uniformly chosen two-dimensional plane is a triangle. They also obtained a generalization of this result to simplices of arbitrary dimension. The probability that a random Gaussian tetrahedron is acute, as well as the distribution of its solid angles, is discussed in the papers of Finch~\cite{finch} and Bosetto~\cite{bosetto}. It seems that the angles of the Gaussian simplex in dimension $d\geq 4$ were not studied so far.
An explicit formula for the solid angles of the regular $d$-dimensional simplex is known and can be found in~\cite{rogers,vershik_sporyshev,absorption}. We shall not rely on this formula. Expected angles of the so-called beta simplices (which contain Gaussian simplices as a limiting case) were used in~\cite{beta_polytopes} to compute the expected $f$-vectors of beta polytopes, but no formula for the expected angles was given there.


The paper is organized as follows. After recalling some necessary facts from convex and stochastic geometry in Section~\ref{sec:facts}, we shall present two different proofs of Theorem~\ref{theo:main} in Sections~\ref{sec:proof1} and~\ref{sec:proof2}. Section~\ref{1548} contains some auxiliary (and probably known) results.


\section{Facts from convex and stochastic geometry}\label{sec:facts}
For vectors $v_1,\dots,v_n\in\R^d$, define their \emph{positive} or \emph{conic hull} as
\[
\pos(v_1,\dots,v_n):=\Big\{\sum_{i=1}^{n}\lambda_i v_i:\lambda_1,\dots, \lambda_n\geq0\Big\}.
\]
A set $C\subset \R^d$ is said to be a \emph{polyhedral cone} (or just a \emph{cone}) if it can be represented as a positive hull of finitely many vectors.
The \emph{solid angle} of the cone $C$ is defined as
\begin{equation}\label{933}
\alpha_d(C):=\P[Z\in C],
\end{equation}
where $Z$  is uniformly distributed on the unit sphere in $\R^d$. The maximal possible value of the solid angle in this normalization is $\alpha_d(\R^d)=1$. If $C\ne\R^d$, then $\P[Z\in C,-Z\in C]=0$ and~\eqref{933} is equivalent to
\begin{equation}\label{1303}
\alpha_d(C) =\frac12\P[W_1\cap C\ne\{0\}],
\end{equation}
where $W_1$ denotes the line passing through $Z$ and $-Z$. Equivalently, $W_1$ is a random $1$-dimensional linear subspace in $\R^d$ uniformly chosen with respect to the Haar measure.

Let $\Span(C)$ be the \emph{linear hull} of $C$, i.e.\ the minimal linear subspace containing $C$. The dimension of the cone $C$, denoted by $\dim C$, is defined as the  dimension of $\Span(C)$. If $\dim C=k<d$, then, by definition, $\alpha_d(C)=0$. However, similarly to~\eqref{933}, we can define $\alpha_k(C)$ as the solid angle of $C$ measured with respect to the linear hull of $C$, which is isomorphic to $\R^k$. Namely, we define $\alpha_k(C):=\P[Z'\in C]$, where $Z'$  is uniformly distributed on the unit sphere in the linear hull of $C$.

If $\dim C=k$ and $C$ is not a $k$-dimensional linear subspace, the conic Crofton formula (see, e.g., \cite[Eq.~(6.63)]{SW08}) implies the following generalization of~\eqref{1303}:
\begin{equation}\label{1626}
\alpha_k(C):=\frac12\P[W_{d-k+1}\cap C\ne\{0\}],
\end{equation}
where $W_{d-k+1}$ denotes a random $(d-k+1)$-dimensional linear subspace in $\R^d$ uniformly chosen with respect to the Haar measure. Alternatively, we can observe that $W_{d-k+1}\cap \Span(C)$ is a random one-dimensional linear subspace of $\Span(C)$ distributed uniformly on the set of all such subspaces, so that~\eqref{1626} follows from~\eqref{1303} applied to $\Span(C)$ as the ambient space.

Let $x_0,\ldots,x_d$ be $d+1$ points in $\R^n$, where $n\geq d$, such that the affine subspace spanned by these points has dimension $d$. A \emph{simplex} $S$ with vertices at $x_0,\dots,x_d$ is defined as the convex hull of these points, that is,
\[
S:=\conv(x_0,\dots,x_d):=\Big\{\sum_{i=0}^d \lambda_i x_i \colon \lambda_0,\ldots,\lambda_d \geq 0,  \sum_{i=0}^d \lambda_i = 1\Big\}.
\]
We say that the dimension of $S$ is $d$.   Define the \emph{solid angle} of $S$ at $x_i$ as
\[
\alpha_d(S,x_i):=\alpha_d(\pos(x_0-x_i, x_1-x_i, \dots, x_d-x_i)).
\]
The sum of the solid angles of $S$ is denoted by
\begin{equation}\label{818}
\gamma_d(S):=\sum_{i=0}^{d}\alpha_d(S,x_i).
\end{equation}
A simplex is called \emph{regular} if the pairwise distances between its vertices are all equal. We shall use the following convenient form of the regular $d$-dimensional simplex in $\R^{d+1}$:
\[
T^d:=\conv(e_0,\dots,e_d),
\]
where $e_0,\dots,e_d$ is the standard orthonormal basis in $\R^{d+1}$.

We shall be interested in random simplices defined as follows. Let $X_0,\ldots, X_d$ be independent random points with standard Gaussian distribution on $\R^d$. The Lebesgue density of any of the $X_i$'s is thus given by
$$
f(x) = (2\pi)^{-d/2} \eee^{-|x|^2/2},
$$
where $|x|$ is the Euclidean norm of $x\in\R^d$. The $d$-dimensional \emph{Gaussian simplex} is defined as the convex hull of $X_0,\ldots,X_d$:
$$
\cP_{d}
:=
\conv(X_0,\ldots,X_d).
$$
With this notation, we can restate our main result as follows:


\begin{theorem}
We have $\E\gamma_d(\cP_d)=\gamma_d(T^d)$.
\end{theorem}

Since the family $(X_0,\ldots,X_d)$ is exchangeable and all solid angles of the regular simplex are equal, an equivalent formulation of the theorem  is as follows:
\begin{equation}\label{943}
\E\alpha_d(\cP_d,X_0)=\alpha_d(T^d,e_0).
\end{equation}
In the next two sections, we  give two  different proofs of~\eqref{943}.


\section{Proof I: Lifting the dimension}\label{sec:proof1}
The main idea is to represent the $d$-dimensional Gaussian simplex in $\R^d$ as a projection of a $d$-dimensional Gaussian simplex in $\R^n$ and then let $n\to\infty$. We shall show that the expected solid angles of both simplices  are equal and there is a ``freezing phenomenon'': In the large $n$ limit, the $d$-dimensional Gaussian simplex in $\R^n$ converges to the regular one.

Consider $d+1$ independent sequences of independent standard Gaussian variables (constructed on the same probability space):
\begin{align*}
N_{01},N_{02},&\dots,N_{0n},\dots,\\
N_{11},N_{12},&\dots,N_{1n},\dots,\\
&\dots,\\
N_{d1},N_{d2},&\dots,N_{dn},\dots.
\end{align*}
For all $n\in\mathbb N$ and $k=0,\dots, d$, let $X_k^{(n)}$ be a standard Gaussian vector in $\R^n$ formed by the first $n$ variables of the $k$th sequence:
\[
X_k^{(n)}:=(N_{k1},\dots,N_{kn})^\top.
\]
For $n\geq d$, the convex hull
\[
\cP_d^{(n)}:=\conv(X_0^{(n)},\dots,X_d^{(n)})
\]
is a $d$-dimensional simplex in $\R^n$, with probability one. In particular, $\cP_d^{(d)}$ is equidistributed with $\cP_d$. We now show that the expected solid angles of $\cP_d^{(n)}$ and  $\cP_d$ are equal.

\begin{lemma}\label{716}
For all $n\geq d$,
\[
\E\alpha_d(\cP_d^{(n)}, X_0^{(n)})=\E\alpha_d(\cP_d, X_0).
\]
\end{lemma}
\begin{proof}
By~\eqref{1626},
\begin{align*}
\E\alpha_d(\cP_d^{(n)}, X_0^{(n)})&=\E \alpha_d (\pos(X_1^{(n)}-X_0^{(n)},\dots,X_d^{(n)}-X_0^{(n)}))\\
&=\frac12\P[W_{n-d+1}\cap \pos(X_1^{(n)}-X_0^{(n)},\dots,X_d^{(n)}-X_0^{(n)})\ne\{0\}],
\end{align*}
where $W_{n-d+1}$ is the random $(n-d+1)$-dimensional linear subspace of $\R^n$ distributed uniformly on the set of all such subspaces and independent of everything else.
Let $e_1,\dots,e_n$ denote the standard orthonormal basis in $\R^n$. Since the standard Gaussian distribution is rotationally invariant, we can replace $W_{n-d+1}$ by $\Span(e_d,\dots,e_n)$, the linear hull of $e_d,\dots,e_n$:
\[
\E\alpha_d(\cP_d^{(n)}, X_0^{(n)})=\frac12\P[\Span(e_d,\dots,e_n)\cap \pos(X_1^{(n)}-X_0^{(n)},\dots,X_d^{(n)}-X_0^{(n)})\ne\{0\}].
\]
The next observation is that
\[
\Span(e_d,\dots,e_n)\cap \pos(X_1^{(n)}-X_0^{(n)},\dots,X_d^{(n)}-X_0^{(n)})\ne\{0\}
\]
if and only if the convex hull of the orthogonal projection of $X_1^{(n)}-X_0^{(n)},\dots,X_d^{(n)}-X_0^{(n)}$ on
\[
\Span(e_d,\dots,e_n)^\perp=\Span(e_1,\dots,e_{d-1})
\]
contains the origin. By definition, the orthogonal projection of $X_k^{(n)}$ on $\Span(e_1,\dots,e_{d-1})$ is $X_k^{(d-1)}$. Therefore,
\begin{equation}\label{851}
\E\alpha_d(\cP_d^{(n)}, X_0^{(n)})=\frac12\P[0\in \conv(X_1^{(d-1)}-X_0^{(d-1)},\dots,X_d^{(d-1)}-X_0^{(d-1)})].
\end{equation}
This relation holds for all $n\geq d$ and the right-hand side does not depend on $n$. Thus,
\[
\E\alpha_d(\cP_d^{(n)}, X_0^{(n)})=\E\alpha_d(\cP_d^{(d)}, X_0^{(d)})=\E\alpha_d(\cP_d, X_0),
\]
which proves the lemma.
\end{proof}
To complete the proof of~\eqref{943}, we let $n\to\infty$.
It follows from the strong law of large numbers that  for all $0\leq i< j\leq d$,
\[
\lim_{n\to\infty} \frac{\langle X_i^{(n)},X_j^{(n)}\rangle}{n}=0\quad\text{and}\quad\lim_{n\to\infty} \frac{\langle X_i^{(n)},X_i^{(n)}\rangle}{n}=1,
\quad
\text{a.s.},
\]
which implies that for all $1\leq i< j\leq d$,
\[
\lim_{n\to\infty} \frac{\langle X_i^{(n)}-X_0^{(n)},X_j^{(n)}-X_0^{(n)}\rangle}{|X_i^{(n)}-X_0^{(n)}| |X_j^{(n)}-X_0^{(n)}|}=1/2\quad\text{a.s.}
\]
On  the other hand, for the regular simplex $T^d=\conv(e_0,\dots,e_d)$ we have
\[
 \frac{\langle e_i-e_0, e_j-e_0\rangle}{|e_i-e_0| |e_j-e_0|}=1/2.
\]
By Corollary~\ref{533} and Remark~\ref{529} stated below, this yields the convergence of the corresponding solid angles:
\[
\lim_{n\to\infty}\alpha_d(\cP_d^{(n)}, X_0^{(n)})=\alpha_d(T^d,e_0)\quad\text{a.s.}
\]
Since the solid angle is bounded by $1$, the dominated convergence theorem implies that
\[
\lim_{n\to\infty}\E\alpha_d(\cP_d^{(n)}, X_0^{(n)})=\alpha_d(T^d,e_0).
\]
Applying  Lemma~\ref{716} completes the proof.

\section{Proof II: Projection} \label{sec:proof2}
The starting point of our second proof of Theorem~\ref{theo:main} is the identity
\begin{equation}\label{444}
\E\alpha_d(\cP_d,X_0)=\frac12\P[0\in\conv(Y_1-Y_0,\dots,Y_d-Y_0)],
\end{equation}
where $Y_0,\dots, Y_d$ are independent standard Gaussian vectors in $\R^{d-1}$. Even though this identity follows from~\eqref{851}, we provide an independent argument.
With probability one,  the cone $\pos(X_1-X_0,\dots,X_d-X_0)$ is of full dimension $d$ and does not coincide with $\R^d$. Therefore, by~\eqref{1303},
\[
	\E\alpha_d(\cP_d,X_0)=\frac12\P[W_1\cap\pos(X_1-X_0,\dots,X_d-X_0)\ne\{0\}],
\]
where $W_1$ is a uniformly distributed one-dimensional linear subspace of $\R^d$ which  is  independent of $X_0,\dots,X_d$. By rotational invariance, we can replace $W_1$ by the line $\Span(e)$, where $e\in\R^d$ is any unit vector.
Let $Y_1,\ldots,Y_d$  be the projections of $X_1,\dots,X_d$ on the orthogonal complement of $e$ (which we identify with $\R^{d-1}$).
The key observation is that
\[
\Span(e) \cap\pos(X_1-X_0,\dots,X_d-X_0)\ne\{0\}\quad\text{if and only if}\quad 0\in\conv(Y_1-Y_0,\dots,Y_d-Y_0).
\]
The proof of~\eqref{444} is complete.


Let us now look at the right-hand side of~\eqref{444}. Observe  that $0\in\conv(Y_1-Y_0,\dots,Y_d-Y_0)$ if and only if there exist $\lambda_1,\dots,\lambda_d\geq0$ with $\lambda_1+\dots+\lambda_d>0$ such that
\[
\lambda_1(Y_1-Y_0)+\dots+\lambda_d(Y_d-Y_0)=0,
\]
or, equivalently,
\begin{equation}\label{1059}
(-\lambda_1-\dots-\lambda_d)Y_0+\lambda_1Y_1+\dots+\lambda_dY_d=0.
\end{equation}
Consider a $(d-1)\times (d+1)$-matrix $Y$ whose columns are $Y_0,\dots,Y_d$:
\[
Y:=(Y_0,\dots, Y_d).
\]
 Condition~\eqref{1059} is equivalent to
\begin{equation}\label{1115}
Y\left(\begin{array}{c}
-\lambda_1-\dots-\lambda_d\\
\lambda_1\\
\dots\\
\lambda_d
\end{array}\right)
=0
\quad\text{or}\quad \left(\begin{array}{c}
-\lambda_1-\dots-\lambda_d\\
\lambda_1\\
\dots\\
\lambda_d
\end{array}\right)\in\ker Y.
\end{equation}
Now consider the cone $C\subset \R^{d+1}$ defined as
\[
C:=\pos(e_1-e_0,\dots,e_d-e_0),
\]
where, as above, $e_0,\dots,e_d$ is the standard orthonormal basis in $\R^{d+1}$.
By definition,
\begin{equation}\label{1333}
\alpha_d(C)=\alpha_d(T^d,e_0).
\end{equation}
On the other hand,  we obviously have
\[
C=\{(-\lambda_1-\dots-\lambda_d,\lambda_1,\dots,\lambda_d)\in\R^{d+1}:\lambda_1,\dots,\lambda_d\geq0\}.
\]
Therefore, the condition that there exist $\lambda_1,\dots,\lambda_d\geq0$ with $\lambda_1+\dots+\lambda_d>0$ such that~\eqref{1115} holds is equivalent to $C\cap \ker Y\ne\{0\}$. This yields
\[
\E\alpha_d(\cP_d,X_0)=\frac{1}{2}\P[C\cap \ker Y\ne\{0\}].
\]
By definition, $Y$ is a $(d-1)\times(d+1)$ matrix whose entries are independent standard Gaussian variables. Thus, with probability one, $\ker Y$ is a $2$-dimensional linear subspace in $\R^{d+1}$ and it is uniformly distributed on the set of all  $2$-dimensional subspaces in $\R^{d+1}$ with respect to the Haar measure.
Recall that  $\Span (C)$ denotes the minimal linear subspace containing $C$.
Since $\dim C=d$, we have that $W_1':= \ker Y\cap \Span(C)$ is uniformly distributed on the set of all  $1$-dimensional linear subspaces in $\Span(C)$ with respect to the Haar measure. Therefore,
\begin{align*}
\E\alpha_d(\cP_d,X_0)&=\frac{1}{2}\P[C\cap \ker Y\ne\{0\}]=\frac{1}{2}\P[C\cap(\Span(C)\cap \ker Y)\ne\{0\}]\\
&=\frac{1}{2}\P[C\cap W_1'\ne\{0\}]=\alpha_d(C),
\end{align*}
see~\eqref{1303} for the last equality.
Together with~\eqref{1333}, this completes the proof of~\eqref{943}.

\section{Appendix}\label{1548}

\subsection{Formula for the solid angle of a simplicial cone} \label{subsec:cont}
Since we were not able to find a precise reference for the following statement, we present its proof which was obtained jointly with Anna Gusakova.
\begin{proposition}\label{528}
For linearly independent vectors $v_1,\dots,v_d\in\R^d$, consider the cone
\[
C:=\pos(v_1,\dots,v_d).
\]
Then, the solid angle of $C$ is given by
\[
\alpha_d(C)=\frac{\sqrt{\det \Gamma}}{(2\pi)^{d/2}}\int_{\R^d_+}\exp\bigg(- \frac 12 \langle x,\Gamma x\rangle\bigg)\,\dd x,
\]
where $\Gamma$ is the Gram matrix of  $v_1,\dots,v_d$.
\begin{proof}
%
Let $V$ be the $d\times d$-matrix whose columns are $v_1,\ldots, v_d$.
Let $V_{ij}$ denote the $(i,j)$-minor of $V$ obtained by eliminating the $i$th row and the $j$th column.	For $k=1,\dots,d$ consider a vector $n_k$ defined by
\[
n_k:= \frac 1 {\det V}\sum_{i=1}^{d}(-1)^{k+i}(\det V_{ik})e_i,
\]
where  $e_1,\dots,e_d$ is the standard orthonormal basis in $\R^d$. In the following, we shall compute the Gram matrix of $n_1,\dots,n_d$ and show that $C$ can be represented as
\begin{equation}\label{eq:C_form}
C=\{x\in\R^d:\langle n_k, x\rangle\geq0 \text{ for all } k=1,\dots,d\}.
\end{equation}
By definition of $n_k$, we have
\[
\langle n_k, n_l\rangle= \frac 1 {(\det V)^2}\sum_{i=1}^{d}(-1)^{k+l}\det V_{ik}\det V_{il}.
\]
The well-known formula formula for the inverse of a matrix, namely
\[
V^{-1}= \frac 1 {\det V}  ((-1)^{i+j} \det V_{ji})_{i,j=1}^d,
\]
yields the Gram matrix of $n_1,\dots,n_d$:
\begin{equation}\label{241}
\Sigma:=(\langle n_k, n_l\rangle)_{k,l=1}^d=(V^T V)^{-1}= \Gamma^{-1}.
\end{equation}
Let us now prove~\eqref{eq:C_form}. For a vector $x\in\R^d$, let $V_k(x)$ denote the matrix with columns $v_1,\dots,v_{k-1},x,v_{k+1},\dots,v_d$. By the Laplace formula for the determinant, we have
\[
\langle n_k, x\rangle= \frac {\det V_k(x)}{\det V}.
\]
Taking $x=v_i$  gives
$$
\langle n_k, v_i \rangle =
\begin{cases}
1, &\text{ if } k=i,\\
0, &\text{ if } k\neq i.
\end{cases}
$$
Therefore, the cones spanned by $v_1,\ldots,v_d$ and $-n_1,\ldots,-n_d$ are polar to each other and, in particular,
\[
C=\{x\in\R^d:\langle n_k, x\rangle\geq0  \text{ for all } k=1,\dots,d\}.
\]

Since the standard Gaussian distribution is rotationally invariant, Definition~\eqref{933} is equivalent to
\[
\alpha_d(C)=\P[X\in C],
\]
where $X$ is a standard Gaussian vector in $\R^d$. It follows from the last two equations that
\[
\alpha_d(C)=\P[\langle n_k, X\rangle\geq0 \text{ for all } k=1,\dots,d].
\]
The random vector $(\langle n_1, X\rangle,\dots,\langle n_d, X\rangle)$ is centered Gaussian with covariance matrix $\Sigma=\Gamma^{-1}$ given by~\eqref{241} because
\[
\E[\langle n_k, X\rangle\langle n_l, X\rangle]=\langle n_k, n_l\rangle.
\]
Using the formula for its density  function completes the proof.
\end{proof}
\end{proposition}

\begin{corollary}\label{533}
Let $C_0,C_1,\dots, C_n,\dots$ be a sequence of cones in $\R^d$ defined by
\[
C_n:=\pos(v_{n1},\dots,v_{nd}),\quad n=0,1,\dots,
\]
where $v_{01},\ldots,v_{0d}$ are linearly independent.  If for all $1\leq i< j\leq d$
\[
\lim_{n\to\infty}\frac{\langle v_{ni},v_{nj}\rangle}{{|v_{ni}||v_{nj}|}}=\frac{\langle v_{0i},v_{0j}\rangle}{{|v_{0i}||v_{0j}|}},
\]
then
\[
\lim_{n\to\infty}\alpha_d(C_n)=\alpha_d(C_0).
\]
\end{corollary}
\begin{proof}
If we replace each $v_{ni}$ by $v_{ni}/|v_{ni}|$, the solid angles do not change. After this, the statement readily follows from Proposition~\ref{528} and the dominated convergence theorem.
\end{proof}
\begin{remark}\label{529}
Although we stated Corollary~\ref{533} for cones of full dimension, it continues to hold for $d$-dimensional cones of the form $C_n=\pos(v_{n1},\dots,v_{nd})\subset \R^{m(n)}$, where $m(n)\geq d$. Indeed, the solid angles $\alpha_d(C_n)$ depend on the Gram matrix only and do not depend on the ambient space.
\end{remark}

\subsection{Bounds on the sum of the solid angles of a simplex}
A simplex is called non-degenerate if it has non-empty interior.
\begin{proposition}\label{730}
For every non-degenerate simplex $S\subset\R^d$, where $d\geq3$, we have	
\begin{equation}\label{733}
	0<\gamma_d(S)<\frac12.
\end{equation}
Moreover, for every $h\in(0,1/2)$ there exists non-degenerate simplex $S$ such that $\gamma_d(S)=h$.
\end{proposition}
This fact must be  well-known, but we were not able to find an exact reference. For the reader's convenience, we present a proof here. The idea of the proof is due to Sergei Ivanov~\cite{sI18}.
\begin{proof}
First we show that~\eqref{733} holds. 	The lower bound on $\gamma_d(S)$ is trivial. Let us prove the upper one.
	
	Any non-degenerate $d$-dimensional simplex can be represented as an intersection of $d+1$ closed half-spaces in $\R^d$. Namely, there  exist vectors $y_0,\dots,y_d\in\R^d$ and closed half-spaces $H_0^+,\dots,H_d^+$ with boundaries  $H_0,\dots,H_d$ passing through the origin  such that
	\[
	S=\bigcap_{i=0}^d(y_i+H_i^+).
	\]
For $k=0,\dots,d$, we denote by $H_k^-$ the half-space complementary  to $H_k^+$, that is, the closure of $\R^d\setminus H_k^+$.
	
Since $S$ is non-degenerate, the linear hyperplanes $H_0,\dots,H_d$ are in general position, that is, any $d$ of them have linearly independent normal vectors. By Schl\"afli's formula~\cite{lS50},  the hyperplanes divide $\R^d$ into $m$ polyhedral cones $D_1,\dots, D_m$, where \begin{equation}\label{831}
	m=2^{d+1}-2.
	\end{equation}
	By construction,
	\begin{equation}\label{808}
	\Int D_l\cap \Int D_{l'}=\emptyset\quad\text{for}\quad l\ne l',
	\end{equation}
	where $\Int(\cdot)$ denotes the interior of a set. Therefore,
	\begin{equation}\label{1051}
	\sum_{l=1}^{m}\alpha_d(D_l)=1.
	\end{equation}
Each $D_l$ has the following form:
	\[
	D_l=\bigcap_{i=0}^d H_i^{\epsilon_i}:=D^\epsilon,\quad\text{where}\quad\epsilon=(\epsilon_0,\dots,\epsilon_d)\in\{+,-\}^{d+1}.
	\]
Moreover, if $D^\epsilon\ne\{0\}$ for some $\epsilon=(\epsilon_0,\dots,\epsilon_d)\in\{+,-\}^{d+1}$, then $D^\epsilon=D_l$ for some $l$.

	For $k=0,\dots,d$, denote by $x_k$ the vertex of $S$ which is opposite to the face contained in $y_k+H_k$, and
	denote by $C_k$ the internal cone at the vertex $x_k$:
	\[
	C_k:=\pos(x_0-x_k,\dots,x_d-x_k).
	\]
	In terms of the half-spaces, $C_k$ can be represented as follows:
	\[
	C_k=\bigcap_{i:i\ne k} H_i^+.
	\]
	Since $C_k\subset H_k^-$, we have
	\[
	C_k=D^{(+,\dots,+,-,+,\dots,+)},
	\]
	where all coordinates of the upper index are ``$+$'' except one ``$-$'' on the $k$th place. Similarly,
	\[
	-C_k=D^{(-,\dots,-,+,-,\dots,-)}.
	\]
	Thus it follows from~\eqref{818} and~\eqref{1051} that
	\begin{align*}
	\gamma_d(S)&=\sum_{i=0}^{d}\alpha_d(C_i)=\frac12\sum_{i=0}^{d}(\alpha_d(C_i)+\alpha_d(-C_i))<\frac 12 \sum_{l=1}^{m}\alpha_d(D_l)=\frac12.
	\end{align*}
	The intermediate  inequality is strict because~\eqref{831} implies that $m>2d+2$ for $d\geq3$, which means that there exists $D_l$ (with $\alpha_d(D_l)>0$) such that $D_l\ne C_i$ and $D_l\neq -C_i$ for all $i=0,\dots,d$.

Now let us prove the second part of Proposition~\ref{730}.	
	Let  $e_1,\dots,e_d$ be the standard orthonormal basis in $\R^{d}$.
Consider the simplex
	\begin{align*}
	S_0:=\conv(0,e_1,\dots,e_d).
	\end{align*}
Moreover, consider the following two families of simplices indexed by $t\in[0,1)$:
	\[
	S_1(t):=\conv\Big(0,e_1,\dots,e_{d-1}, (1-t)e_d + t (e_1+\ldots+e_{d-1})\Big)
	\]
	and
	\[
	S_2(t):=\conv\Big(0,e_1-t\cdot\frac{e_1+\dots+e_d}{d},\dots,e_d-t\cdot\frac{e_1+\dots+e_d}{d}\Big).
	\]
We have $S_1(0)=S_2(0)=S_0$	and
\[
\lim_{t\to 1-}\gamma_d(S_1(t))=0,\quad\lim_{t\to 1-}\gamma_d(S_2(t))=\frac12.
\]
Pasting both families together, we obtain a continuous family of simplices whose angle sums change from $0$ to $1/2$.  By continuity (see Section~\ref{subsec:cont}), this completes the proof.	
\end{proof}

\section{Acknowledgement}
The authors are grateful to Anna Gusakova for her collaboration  in proving Proposition~\ref{528} and to Sergei Ivanov for communicating to us the idea of proof of Proposition~\ref{730}.

\bibliographystyle{plainnat}

\end{document}